\documentclass[reqno]{amsart}
\usepackage[utf8]{inputenc}
\usepackage[margin=1in]{geometry}
\usepackage[usenames]{color}
\usepackage{amsmath,amsthm,pdfsync,verbatim,graphicx,epstopdf,enumerate,amssymb}
\usepackage[colorlinks=true]{hyperref}
\usepackage{cancel}
\usepackage[framemethod=tikz]{mdframed}

\hypersetup{allcolors=blue}
\numberwithin{equation}{section}

\newcommand{\lb}{\left(}

\newcommand{\rb}{\right)}
\newcommand{\PD}{\partial}

\newcommand{\Beq}{\begin{equation}}
	\newcommand{\Eeq}{\end{equation}}
\newcommand{\beq}{\begin{equation*}}
	\newcommand{\eeq}{\end{equation*}}
\newcommand{\bal}{\begin{align}}
	\newcommand{\eal}{\end{align}}

\newcommand{\n}{\nabla}

\usepackage{mathtools}

\newcommand{\bp}{\begin{prob}}
	\newcommand{\ep}{\end{prob}}
\newcommand{\bpr}{\begin{proof}}
	\newcommand{\epr}{\end{proof}}
\renewcommand{\o}{\omega}

\newcommand{\tred}[1]{{\color{red}{#1}}}

\newcommand{\bel}[1]{\begin{equation}\label{#1}}
	\newcommand{\ee}{\end{equation}}

\newcommand{\rthree}{\mathbb{R}^3}
\newcommand{\zz}{\mathbb{Z}}

\newtheorem{theorem}{Theorem}[section]

\newtheorem{example}[theorem]{Example}

\theoremstyle{definition}
\newtheorem{definition}[theorem]{Definition}

\newtheorem{remark}[theorem]{Remark}

\newcommand{\R}{\mathbb{R}}
\newcommand{\D}{\mathrm{d}}

\newcommand{\Rb}{\mathbb{R}}
\newcommand{\Dc}{\mathcal{D}}

\renewcommand{\th}{^\text{th}}

\usepackage[usenames]{color}
\usepackage{amsmath,pdfsync,verbatim,graphicx,epstopdf,enumerate}

\newcommand{\Ec}{\mathcal{E}}

\newcommand{\Rc}{\mathcal{R}}

\renewcommand{\o}{\omega}

\newcommand{\paren}[1]{\left(#1\right)}

\newcommand{\sparen}[1]{\left\{#1\right\}}

\newcommand{\vw}{w}
\newcommand{\vx}{x}
\newcommand{\vy}{y'}
\newcommand{\vz}{z}
\newcommand{\vs}{s}
\newcommand{\vr}{r}
\newcommand{\vb}{b}

\newcommand{\st}{\,:\,}
\newcommand{\rr}{\mathbb{R}} 
\newcommand{\rn}{\mathbb{R}^n} 
\newcommand{\rnm}{\mathbb{R}^{n-1}}

\newcommand{\zero}{0}

\newcommand{\hess}{\operatorname{Hess}}

\newcommand{\vxt}{{\vx_T}}
\newcommand{\vxtt}{{\vx^T_T}}
\newcommand{\diag}{\operatorname{Diag}}

\title{Microlocal analysis of  Radon transforms over quadric surfaces}
\author[Ambartsoumian, Felea, Krishnan, Nolan and Quinto]{Gaik Ambartsoumian$^{\dagger}$, Raluca Felea$^{\diamond}$, Venkateswaran P.\ Krishnan$^{\flat}$, Clifford J.\ Nolan$^{\sharp}$ and Eric Todd\ Quinto$^{\ast}$}
\address{$^{\dagger}$Department of Mathematics, University of Texas at Arlington, USA E-mail: {\tt gambarts@uta.edu}
\newline
\indent $^{\diamond}$School of Mathematics and Statistics, Rochester Institute of Technology, USA E-mail: {\tt rxfsma@rit.edu}
\newline
\indent $^{\flat}$Tata Institute of Fundamental Research - Centre for Applicable Mathematics, Bangalore, India
\newline
\indent\: E-mail: {\tt vkrishnan@tifrbng.res.in}
\newline
\indent $^{\sharp}$Department of Mathematics and Statistics, University of Limerick, Ireland E-mail: {\tt clifford.nolan@ul.ie}
\newline
\indent $^{\ast}$Department of Mathematics, Tufts University, USA, E-mail: {\tt todd.quinto@tufts.edu}}
\begin{document}

\maketitle

\begin{abstract}
    We study the microlocal properties of generalized Radon transforms over a family of quadric hypersurfaces whose centers lie on an orientable hypersurface $S$. The quadric surfaces we consider are level sets of the quadratic form associated to a symmetric, invertible matrix $A$, 
    with real entries.  
    We study the singularities of the right and left projections of the canonical relation associated with these operators and  show that  they are determined by the signature of the matrix $A$ and the hypersurface $S$. If the matrix is positive/negative definite (i.e., the surface of integration is an ellipsoid) and $S$ is strictly convex, we prove that the singularities are folds. If the matrix is indefinite (i.e., the surface of integration is a hyperboloid-type quadric) and $S$ is either strictly convex or a cylinder, then cusp, fold, or blowdown singularities are present. We also study the case when the surface of integration is a paraboloid and show that the Bolker condition is satisfied.
\end{abstract}
\section{Introduction}\label{sec:Intro}
We investigate the microlocal properties of some generalized Radon transforms that map a scalar function in $\mathbb{R}^n$ to an $n$-dimensional set of its integrals over quadric hypersurfaces. Such transforms play a significant role in various applications related to imaging, remote sensing, security screening, and industrial non-destructive testing. Transforms integrating over hyperspheres arise in mathematical models of thermo- and photo-acoustic tomography \cite{kuch-kun-tat-pat, xu2005universal}, mono-static setup of ultrasound reflection tomography \cite{amb-kun, norton-2D, norton-3D}, radar and sonar \cite{cheney-radar, louis-quinto-sonar}. Bistatic configurations of these modalities lead naturally to elliptical and ellipsoidal Radon transforms \cite{ABKQ-ultrasound, AFKNQ-SAR-I, AFKNQ-SAR-II}. Hyperbolic and parabolic Radon transforms are widely used in seismic data processing and geophysics \cite{chen2021fast, gu2009radon, hu2013fast, kaur2020separating, MA2025163, schonewille2001parabolic}. 

The full family of each such hypersurface possesses more than $n$ degrees of freedom. Consequently, studies of the associated generalized Radon transforms often diverge depending on the choice of an $n$-dimensional subfamily of integration surfaces. For instance, a couple of variants of the spherical Radon transform in $\mathbb{R}^n$ have been investigated, including transforms integrating over all hyperspheres passing through a prescribed point \cite{Cormack-Quinto, quinto1982null}, or having centers constrained to lie on another hypersurface (e.g. see \cite{ambartsoumian2010inversion, ambartsoumian2018image, aramyan2024recovering, FHR-even, FRP-odd, haltmeier2014universal, haltmeier2017spherical, kunyansky2007explicit, norton-3D, xu2005universal} and the references therein). The latter type of restriction arises frequently in applications, since the centers of the integration surfaces typically represent the locations of emitters and/or receivers of measurement signals.

The central problems in the study of all generalized Radon transforms concern their injectivity, inversion, stability, and range characterizations. Most of these questions have been addressed for spherical transforms in a limited number of specific data-acquisition geometries, for example, when the centers of the integration surfaces lie on planes, spheres, or cylinders (see the references above, as well as \cite{agranovsky2007range, AQ, agrawal2023simple, agrawal2025simple, ambartsoumian2005injectivity, ambartsoumian2006range} and the references therein). In contrast, transforms integrating over other quadric surfaces remain comparatively under-explored. Deriving explicit closed-form inversion formulas for such operators appears to be difficult. A more attainable goal is to obtain approximate inversion formulas based on suitably filtered versions of the corresponding normal operator. Achieving this, however, requires a detailed understanding of the microlocal properties of the generalized Radon transforms under consideration. A few relevant studies have been carried out by various authors, but only in a limited number of special cases \cite{ABKQ-ultrasound, AFKNQ-SAR-I, AFKNQ-SAR-II, felea2013microlocal, FKNQ-common, FN:fold/cusp,  grathwohl2018microlocal, grathwohl2020imaging, KLQ-microloc-ell, nguyen2019microlocal, stefanov2013curved, WHQ}.

The goal of this paper is to undertake a systematic analysis of the microlocal properties of generalized Radon transforms that integrate a compactly supported function $f$ in $\mathbb{R}^n$ over quadric hypersurfaces whose centers lie on a prescribed hypersurface $S$. The quadric surfaces of integration are defined via an $n\times n$ symmetric, invertible matrix $A$ with real entries.  We study the singularities of the right and left projections of the canonical relation associated with these operators. In general, a canonical relation describes the relation between the singularities of the data and the singularities of the image.  We show that the hypersurface $S$ and the signature of the matrix $A$ determine the singularities of the projections (i.e., folds, cusps, or blowdowns). If the matrix is positive or negative definite (i.e., the surface of integration is an ellipsoid) and $S$ is strictly convex, we prove that the singularities are folds. If the matrix is indefinite (i.e., the surface of integration is a hyperboloid-type quadric) and $S$ is either strictly convex or a cylinder, then cusp, fold, or blowdown singularities are present. 

We also study the case when the surface of integration is a paraboloid and show that in that case, the Bolker condition is satisfied. Our study takes its initial motivation from the paper \cite{WHQ} in which several of our transforms are studied on restricted domains for which the canonical left and right projections from the canonical relation do not drop rank; thus, no added singularities/artifacts are present in the normal operator.

The rest of the article is organized as follows. In Section \ref{sec:Prelim} we recall the necessary definitions and introduce the notation used throughout the paper. Section 
\ref{sec:A} introduces the generalized Radon transform studied here and describes the associated canonical relation. In Section \ref{sec:definite} we consider transforms corresponding to positive or negative definite matrices $A$. Section \ref{sec:evals-mixed-sign} addresses the case in which $A$ has eigenvalues of mixed sign. In Section \ref{sec:cylindrical-geom} we analyze the transforms in the cylindrical geometry of data acquisition. Section \ref{sec:paraboloids} treats paraboloids and more general surfaces of integration. Section \ref{sec:normal-operator} focuses on the associated normal operators. Finally, Section \ref{sec:acknowledgements} contains our acknowledgements.

%%%%%%%%%%%%%%%%%%%%%%%%%%%%%%%%%%%%%%%%%%%%%%%

\section{Preliminaries} \label{sec:Prelim}

In this section, we define the types of singularities that come up when analyzing our transforms, and we define some of the basic terms.

 Let $f$ be a smooth function $f : {\mathbb{R}}^n \to {\mathbb{R}}^n$.  
 
  Let $S_1(f)$ be the following smooth hypersurface:
\[S_1 (f) =\{ p \in {\mathbb{R}}^n: \det(\D f_p)=0, \ \D (\det (\D f))_p\neq 0 \}.
 \] 
 We say that the derivative of $f$ drop ranks simply by 1 at points $p\in S_1(f)$.

\begin{definition}\cite{golubitsky.guillemin.73} The function $f$ has a  fold  singularity along $ S_1(f)$  if for all $p\in S_1(f)$,    $\ker(\D f_p)$ is transversal to $T_p S_1(f)$.
\end{definition}
Therefore, if $f$ has a fold singularity along $S_1(f)$, we may choose local coordinates such that  there exists a smooth $h:\mathbb R^{n}\to\mathbb R$, with $f(x_1, x_2, \dots, x_n)=(x_1, x_2, \dots, x_{n-1},h(x))$ and we have the following local description of $S_1(f)$:
\[ 
S_1(f)=\left\{(x_1,\dots,x_n)\in\mathbb R^n: \frac{\partial h}{\partial x_n} (x)=0, \ \D \left(\frac{\partial h}{\partial x_n} \right) \lb x\rb\neq 0\right\}.
\]

The fold condition implies $\frac{\partial^2 h}{\partial^2 x_n}  \lb x\rb\neq 0 $.

Any map with fold singularities can be put in the local normal form   $f(x_1,x_2, \dots, x_n)=(x_1, x_2, \dots, x_n^2)$ \cite{golubitsky.guillemin.73}.

\begin{definition}\cite{golubitsky.guillemin.73} The function $f$ has a  blowdown  singularity along $ S_1(f)$  if for all $p\in S_1(f)$,    $\ker (\D f_p)\subset T_p S_1(f)$.
\end{definition}
Any map with blowdown singularities can be put in the normal form \cite{golubitsky.guillemin.73}:   
\[
f(x_1,x_2, \dots, x_n)=(x_1, x_2, \dots, x_{n-1},x_{n-1}x_n).
\]
Next, we define cusp singularities. We consider a point $p\in S_1(f)$ where  $\ker(\D f)_p \subset T_pS_1(f)$. Consider a vector field $v$ along $S_1(f)$ such that $v \in  \ker(\D f)$ at all points on $S_1(f)$. Let $K$ be a smooth function such that $K|_{S_1(f)}=0$ and $\D K|_{S_1(f)} \neq 0$. Then, since $\D K$ annihilates vectors tangent to $S_1(f)$, the function $v(K)=\D K(v)$ has a zero at $p$.

\begin{definition} \cite{golubitsky.guillemin.73} The function $f$  has a cusp singularity at $p$ if  $\D K(v)$ has a simple zero at $p$.
\end{definition}
Without loss of generality, assume that $p=0$ is a cusp point. Using suitable coordinates defined near $p=0$, we may write $f(x_1, x_2, \dots, x_n)=(x_1, x_2, \dots x_{n-1}, h(x))$, with $h(0)=0$, and locally, $S_1(f)=\{(x_1,\dots,x_n)\in\mathbb R^n: \frac{\partial h}{\partial x_n} (x)=0, \D( \frac{\partial h}{\partial x_n}) (x) \neq 0\}$. We say that $f$ has a cusp singularity at $p=0$ if $\frac{\partial^2 h}{\partial x_n^2} (0)=0, \ \ \frac{\partial^3 h}{\partial x_n^3} (0) \neq 0$  and the matrix $\left[ \D (\frac{\partial h}{\partial x_n})(0)\ \ \D (\frac{\partial^2 h}{\partial x_n^2})(0) \right]$ 
has rank $2$. More generally, if $\ker \D  f=\frac{\partial}{\partial x_n}$ is tangent simply to $S_1(f)$ along
\[
S_{1,1}(f)=\Big{\{} x: \frac{\partial h}{\partial x_n} (x)=\frac{\partial^2 h}{\partial x_n^2} (x)=0  
\ \mbox{and}\ \D  \lb \frac{\partial h}{\partial x_n}\rb (x), \D  \lb \frac{\partial^2 h}{\partial x_n^2}\rb (x),\ \mbox{are linearly independent}
\Big{\}},
\]
then $S_{1,1}(f)$ is a codimension-$2$ submanifold, called the cusp set of $f$. Any map with a cusp singularity can be put into a local normal
form \cite{golubitsky.guillemin.73}: 
\[
f(x_1,x_2, \dots, x_n)=(x_1,x_2, \dots, x_{n-1}, x_{n-1}x_n+x_n^3).
\]

\begin{definition}
 Let $A$ be an invertible symmetric  $n \times n$ matrix. We say $A$ 
  is positive definite if all the eigenvalues of $A$ are positive and negative definite if all the eigenvalues of $A$ are negative. When $A$ has both positive and negative eigenvalues, $A$ is indefinite. 
\end{definition}

\begin{definition}
An orientable hypersurface $S$ is called strictly convex if its second fundamental form is definite at every point of $S$ (with respect to a chosen orientation), i.e., its principal curvatures are strictly positive (or negative) at every point (e.g., see \cite{schneider1972closed}).
\end{definition}

\begin{remark}
   Let \( S = \{(x,q(x))\} \) be the graph of a \( C^2 \) function. The matrix \( \mathrm{Hess}(q) \) is everywhere positive (or negative) definite if and only if \( S \) is a locally strictly convex hypersurface (up to orientation); see, e.g., \cite{doCarmoDG} or \cite{GilbargTrudinger}.
\end{remark}
\par Let $X$ and $Y$ be open subsets of $\rr^{n_X}$ and $\rr^{n_Y}$ respectively,
and let $\varphi$ be a nondegenerate phase function \cite{D} on $X \times
Y\times (\R^N\setminus 0)$, for some $N\ge 1$. Let $a$ be an amplitude in
$S^{m}_{1,0}$ \cite{D}. 
\begin{definition} \cite{Ho1971}
A Fourier integral operator (FIO) is a continuous linear operator $F:\mathcal E'(Y)\to \mathcal D'(X)$ with the kernel 
$K_F(x,y)=\int_{\R^n} e^{i\varphi(x,y;\theta)} a(x,y;\theta)\, \D \theta$.
\end{definition}
The  order of $F$ is $\mu:=m+\frac{2N-n_X-n_Y}4$,
and   the canonical relation of $F$ is
$$C=\big\{ (x,\D_x\varphi; y, -\D_y\varphi)\, : \, (x,y;\theta)\in \mbox{supp}(a),\, \D_{\theta}\varphi(x,y;\theta)=0\,\big\}\subset (T^*X\setminus 0)\times (T^*Y\setminus 0).$$
We define the canonical left and right projections from $C$ to be 
$\pi_R: C \rightarrow  T^*X, \ \pi_R(x, y,
\theta)=(y,-\D_y\varphi)$ and $\pi_L: C \rightarrow  T^*Y, \  \pi_L(x, y, \theta)=(x,\D_x\varphi)$. We say that a canonical relation $C\subset T^*X\times T^*Y$ satisfies the  Bolker condition if  $\pi_L$ is an injective immersion.

\section{The canonical relation of $\Rc$}  \label{sec:A}

In this section, we provide the basic notation for our transform.  After defining the transform, we calculate its canonical relation and some basic terms.

In the discussion below, we mainly follow the notation from \cite{WHQ}. Let $S$ be a smooth connected hypersurface in $\Rb^n$, representing the centers of integration surfaces of the Radon transforms considered in this paper.  We assume throughout the article that $S$ is a surface parameterized by the smooth function $q:\Omega\to\rr$ for some open $\Omega$ in $\rnm$: \bel{param q} y_n = q(y_1,\dots,y_{n-1}),\ \ \vy=(y_1,\cdots,y_{n-1})\in \Omega.\ee  However, one can often interchange coordinates to locally parameterize a general manifold $S$ by \eqref{param q}, and many of our theorems will be valid in this case.

We let $A$ be an $n\times n$ symmetric, invertible matrix with real entries. We use $A$ to define certain quadric surfaces of integration and assume that the entries of $A$ are constants. In some cases, we assume that $A$ is a diagonal matrix, and in many cases, we will generalize to the case when $A$ is not diagonal. We also adopt the following convention: \[\vx_T=\vx-\vs,\ \ \text{where}\ \ \vs=(\vy,q(\vy))\in S, \ \ \text{for}\ \ \vy\in \Omega,\] and define the function 
 \bel{def:psi} \begin{aligned}\psi(\vy, t,\vx,\omega)&=\o(t-\vx_T^TA\vx_T), \;\;t>0.\\
&=\omega\left(t-(\vx'-\vy, x_n-q(\vy))A(\vx'-\vy, x_n-q(\vy))^T\right).\end{aligned}
\ee where $\vx' = (x_1,\dots, x_{n-1})$ and $\vy \in \Omega$. We suppress the dependence of $\psi$ on $A$ since we assume $A$ is fixed in each theorem. 

The generalized Radon transform that we study is
\bel{def:R}
\Rc f(y',t)=\int\limits_{\Rb} \int\limits_{\Rb^n} |\nabla_\vx \psi|f(\vx) e^{i  \psi(y', t, \vx, \omega) } \D \vx\, \D \omega.
\ee
We assume that $S$ does not intersect the support of $f$. For each
$(\vy,t)\in \Omega\times (0,\infty)$, $\Rc f(\vy,t)$ integrates $f$ over the set $\vxtt A\vxt = t$, which is a level set of the homogeneous second order polynomial $\vxtt A\vxt$. 

 The canonical relation associated to $\Rc$ is then given by
\[
C=\Big{\{} (\vy, \vx_T^T A \vx_T,  \D_{y'} \psi, \omega; \vx, -\D_\vx \psi) : \psi_{\omega}(y',t,x,\o)=0 \Big{\}}.
\]
The variables $(\vy,\vx,\o)$ parametrize the conic Lagrangian submanifold $C$, based on the definition above. 

Let us now compute the terms $\D_{y'} \psi$ and $\D_\vx \psi$ separately. A straightforward computation gives
\[
\D^{T}_{y'} \psi= 2 \o \vx_T^TAB^T
\]
with $B$ being an $(n-1)\times n$ matrix given by 
\[
B=
\begin{pmatrix}
    1 & 0 &\cdots & 0 & \PD_{y_1}q\\
    \vdots&\vdots & \ddots &\vdots & \vdots \\
    0 & 0&\cdots & 1&\PD_{y_{n-1}}q
\end{pmatrix}.
\]
Similarly, another straightforward computation gives
\[
\D_\vx \psi =-2 \o A(\vx'-\vy, x_n-q(\vy))=-2\omega\vx_T^TA.
\]
Now that we have the canonical relation $C$ associated to $\Rc$, let us
study the canonical left ($\pi_L$) and right ($\pi_R$) projections to
the respective cotangent bundles.  We will show that the
projections drop rank simply by $1$ on a codimension 1 submanifold $\Sigma$; see \eqref{Eq}. We will prove that the singularities of the projections are determined by the signature of $A$ and the choice of $S$. We distinguish several cases described in sections \ref{sec:definite}-\ref{sec:paraboloids}. 

\medskip

Now that we have established the preliminaries, we will analyze the transform $\Rc$  when the matrix $A$  and $S$ have different properties.

\section{Case when $A$ is positive definite and $S$ is a smooth strictly convex hypersurface}\label{sec:definite}

\begin{theorem}\label{thm:pos def}  Let  $\Rc: \Ec'(\rr^n\setminus S) \rightarrow \Dc'(\Omega
\times(0,\infty))$ be the Radon transform defined by \eqref{def:psi} and \eqref{def:R}. Assume  $A$ is an invertible positive definite matrix, and
Hess($q$) is definite (either positive or negative definite), then $\Rc$ is an FIO  associated to a two-sided fold canonical relation.
\end{theorem}

Recall that we assume $t>0$ in this article (see \eqref{def:psi}). Note that if $A$ is negative definite, then there are no solutions to $\vxtt A \vxt = t$ for $t>0$, but if we restrict to  $t<0$ then we get the same result as in Theorem \ref{thm:pos def}.

\begin{proof} 
Using the parameterization $(\vy,\vx,\o)$ of the canonical relation $C$, we have 
\[
\pi_{L}(\vy, \vx,\o)= (\vy,  \vx_T^T A \vx_T, 2 \omega \vx_T^T A B^T, \omega).
\]
We next  follow \cite{WHQ} in computing $\D \pi_L$ and $\D \pi_R$. We have %
\[
\D \pi_L= 
\begin{pmatrix} 
I_{n-1} & 0 & 0 \\
-2 \vx_{T}^{T}AB^{T} & 2 \o \vx_T^T A & 0\\
\star &  2\o BA & 2 BA\vx_{T}\\
0 & 0 & 1
\end{pmatrix}.
\]
The quantities denoted by  $\star$ here and henceforth are not consequential to the calculations. 

A straightforward calculation gives 
\bel{Det DPiL}
\det (\D \pi_L)=2\o\det \begin{pmatrix} 
\vx_T^T A \\
BA\\
\end{pmatrix} =2 \o  \det \Bigg{(}\begin{pmatrix}
\vx_T^T \\
B\\
\end{pmatrix} A\Bigg{)} =2\o \det \begin{pmatrix}
\vx_T^T \\
B\\
\end{pmatrix} \det (A).
\ee
Next 
\[
\det \begin{pmatrix}
\vx_T^T \\
B
\end{pmatrix} 
= \det \begin{pmatrix} 
x_1-y_1 & x_2-y_2 & \cdots & x_{n-1}-y_{n-1} & x_n-q(\vy) \\
1 & 0 & \dots & 0 & \frac{\partial q}{\partial y_1}\\
\vdots & \vdots & \ddots & \vdots & \vdots \\
0 & 0 & \dots & 1 & \frac{\partial q}{\partial y_{n-1}}\\
\end{pmatrix}.
\]
By an easy computation, this turns out to be 
\[
\det \begin{pmatrix}
\vx_T^T \\
B
\end{pmatrix} 
= (x_n-q(\vy)-\frac{\partial q}{\partial y_1} (x_1-y_1)-\dots -\frac{\partial q}{\partial y_{n-1}}(x_{n-1}-y_{n-1})=
(-(\nabla q)^T(\vy), 1) \cdot \vx_T^T.
\]

Thus, $\D \pi_L$ drops rank where $\det (\D \pi_L)=0$, that is,  on the codimension 1 hypersurface 
\Beq \label{Eq}  \Sigma=\{ (\vy,\vx,\o):  (-(\nabla q)^T(\vy), 1) \cdot (x_1-y_1,\cdots, x_{n-1}-y_{n-1}, x_n-q(\vy))=0    \}.\Eeq
\par  Since the derivative with respect to $x_n$ is $1$, the determinant drops rank simply by $1$.

Next, we compute  $\ker (\D \pi_L)$ to determine the nature of the singularity of $\pi_L$. Assume that $(\delta \vy, \delta \vx,\delta \o)\in \ker(\D \pi_L)$. Then note that $\delta \vy=\delta \o=0$.  We then have 
\[
\begin{pmatrix} 
\vx_T^T \\
B
\end{pmatrix} 
A\begin{pmatrix} 
\delta x_1 \\
\delta x_2\\
\vdots\\
\delta x_n\\
\end{pmatrix} =0
\]
Let $\vr=A \delta \vx$. 
Expanding this system, 
\[
\begin{pmatrix} 
\vx_T^T \\
B
\end{pmatrix} \vr=
\begin{pmatrix} 
x_1-y_1 & x_2-y_2 & \cdots & x_{n-1}-y_{n-1} & x_n-q(\vy) \\
1 & 0 & \cdots & 0 & \frac{\partial q}{\partial y_1}\\
\vdots & \vdots & \ddots & \vdots & \vdots \\
0 & 0 & \cdots & 1 & \frac{\partial q}{\partial y_{n-1}}
\end{pmatrix} 
\begin{pmatrix} 
r_1 \\
r_2\\
\vdots\\
r_n\\
\end{pmatrix} =0.
\]
Thus
\[ 
\begin{aligned} 
 r_1=- \frac{\partial q}{\partial y_1}  r_n,\quad r_2=- \frac{\partial q}{\partial y_2}  r_n,\quad
 \cdots, \quad r_{n-1}=- \frac{\partial q}{\partial y_{n-1}}  r_n
\end{aligned} 
\]
and \\
\[ r_1 (x_1-y_1)+ r_2(x_2-y_2)+ \cdots+ (x_n-q)r_n=0.
\]
Replacing $r_{i}$ in the last equation, we get  
\Beq\label{DpiL-det-equation}
r_n \Big{\{} (x_n-q) -\frac{\partial q}{\partial y_1} (x_1-y_1)- \dots -\frac{\partial q}{\partial y_{n-1}}(x_{n-1}-y_{n-1})\Big{\}} =0. 
\Eeq
This is equivalent to $r_n \langle -\nabla q, 1\rangle \cdot
\vx_T^T=0$ which is already true since this is the defining equation of $\Sigma$. We will next investigate if the singularity is a fold. Let us assume that we have a $(\delta \vy, \delta \vx,\delta \o)$ in the kernel that is also tangent to the surface $\Sigma$. As before $\delta \vy$ and $\delta \o$ components are $0$. We then have $\langle -\nabla q, 1\rangle \cdot \delta \vx=\langle -\nabla q, 1\rangle  \cdot A^{-1} \vr=r_n \langle-\nabla q, 1\rangle \cdot A^{-1} \langle -\nabla q, 1\rangle^T$. Thus, if $A$ is positive definite or negative definite, the relation above implies $ r_n \langle-\nabla q, 1\rangle \cdot A^{-1} \langle -\nabla q, 1\rangle^T =0$ only if $r_n=0$, which then implies that $r_i=0$ for each $1\leq i\leq n-1$ as well. Hence $\delta x=0$, which implies that $\pi_L$ has a fold singularity. Therefore, if $A$ is positive or negative definite, then $\pi_{L}$ has only fold-type singularities.

Next, we consider the right projection $\pi_R$. From the canonical relation, we have that, 
\[ 
\pi_R(\vy, \vx,\omega)=(\vx, 2\omega \vx^{T}_T A).\]
Thus 
\[
\D \pi_R= \begin{pmatrix} 
\zero_{n\times (n-1)} &  I_n&\zero_{n\times 1} \\
-2\omega AB^T & \star& 2A\vx_T
\end{pmatrix}.
\]
The right projection $\pi_R$ drops rank by $1$ over the same surface
$\Sigma$,  and $(\delta \vy,  \delta \vx,\delta \omega) \in \ker( \D \pi_R)$ if 
\[\delta\vx = \zero,\ \ \text{and }\ \ 
2A \begin{pmatrix} 
-\omega B^T &\vx_T
\end{pmatrix}
\begin{pmatrix} 
\delta \vy \\
\delta \omega\\
\end{pmatrix}=0.
\] 
We then have 
 \[
 \begin{pmatrix} 
 -\omega & 0 & \dots & 0 & x_1-y_1\\
 0 & -\omega & \dots & 0 & x_2-y_2\\
\vdots & \vdots & \ddots & \vdots & \vdots \\
0 & 0 & \dots & -\omega & x_{n-1}-y_{n-1}\\
-\omega\frac{\partial q}{\partial y_1} & -\omega \frac{\partial q}{\partial y_2} & \dots & -\omega \frac{\partial q}{\partial y_{n-1}} & x_n-q\\
\end{pmatrix}
\begin{pmatrix} 
\delta y_1 \\
\delta y_2\\
\vdots \\
\delta y_{n-1}\\
\delta \omega
\end{pmatrix}=0.
\]

Solving for $\delta y_i$, we obtain $\delta y_i=\frac{x_i-y_i}{\omega} \delta \omega$. Hence $\ker(\D \pi_R)$ is given by \[
\delta \o \begin{pmatrix} 
 \frac{x_1-y_1}{\omega}  \\
\frac{x_2-y_2}{\omega} \\
\vdots \\
\frac{x_{n-1}-y_{n-1}}{\omega} \\
1
\end{pmatrix}.
\]When applying this to the defining function of $\Sigma$, it gives
\Beq\label{Eq 1}
\begin{aligned} 
&\lb \frac{x_1-y_1}{\omega} \frac{\partial}{\partial y_1} +  \cdots + \frac{x_{n-1}-y_{n-1}}{\omega} \frac{\partial}{\partial y_{n-1}} +  \frac{\partial}{\partial \omega} \rb \\ &\times \lb x_n-q(y)-\frac{\partial q}{\partial y_1} (x_1-y_1)-\dots -\frac{\partial q}{\partial y_{n-1}}(x_{n-1}-y_{n-1}\rb\\
&=
 \frac{1}{\omega} {\vx'}_T^T \hess (q) \vx'_T \mbox{ where } \vx'=(x_1, x_2, \dots, x_{n-1}).
 \end{aligned} 
 \Eeq
 Since the support of the function $f$ we are interested in recovering and $S$ do not intersect, we have that $\vx_{T}\neq 0$. Let us consider two cases: (i) when $\vx_{T}'=0$ and (ii) when $\vx_{T}'\neq 0$. In the second case, if we assume that $\hess(q)$ is positive or negative definite, then \eqref{Eq 1} is non-zero, showing that $\pi_{R}$ has a fold singularity. In the first case, since we are considering points $(\vy,\vx,\o)$ on $\Sigma$, we have $x_n=q(\vy)=q(\vx')$. But this would mean that we have a point common to the support of the function $f$ as well as the surface $S$. We get a contradiction. Therefore, if we assume that Hess$(q)$ is positive or negative definite, $\pi_R$ has a fold singularity. 
 \end{proof}

\begin{remark}
If Hess$(q)$ is identically zero, then $S$ is a hyperplane, and $\pi_R$ has a blowdown singularity.
\end{remark}

\begin{remark}\label{rem:no intersection} Note that the condition defining $\Sigma$ is just that $\vx$ is in the
 tangent plane to $S$ at $(\vy,q)$. Therefore, if the surfaces of
integration given by \bel{def:T}T(\vy,t) = \sparen{\vx\st \vx^T_T A
\vx_T =t}\ee do not intersect the tangent plane to $S$ at $(\vy,q)$,
then there are no points on $C$ where the projections drop rank. This
is one of the main theorems of \cite{WHQ}. \end{remark}

%%%%%%%%%%%%%%%%%%%%%%%%%%%%%%%%%%%%%%%%%%%%%%%%%

\section{Case when $A$ is indefinite and $S$ is a smooth strictly convex hypersurface}\label{sec:evals-mixed-sign}

We consider the Radon transform $\Rc$ as before, with $A$ an invertible, symmetric matrix with  $k$ positive eigenvalues and $ n-k$ negative eigenvalues.  We distinguish two cases: when $k=1$ and when $k>1$.

In some of the results in this section, we will assume $A$ is diagonal, so we will often use the notation $\diag(a_{11},a_{22},\dots, a_{nn})$ to denote such a matrix.

\subsection{Case $k >1$}\label{sec:k>1 indefinite}
\par The main theorem for this subsection focuses on the case when  $S$ is a smooth strictly convex hypersurface.

\begin{theorem}\label{thm:k>1 + evals}  Let  $\Rc: \Ec'(\rr^n\setminus S ) \rightarrow \Dc'(\Omega
\times(0,\infty))$ be the Radon transform defined by \eqref{def:psi}
and \eqref{def:R}. If $A$ is a real symmetric invertible matrix with
$k\in \{2,3,\dots,n-1\}$ positive eigenvalues and $ n-k$ negative
eigenvalues, and $S$ is a smooth, strictly convex hypersurface, then
$\Rc$ is an FIO associated to a canonical relation with $\pi_L$ having a
cusp singularity and $\pi_R$ having a fold singularity.
\end{theorem}

We will prove this theorem in subsection \ref{GenA-strictly convex}.
To make the proof more accessible,  we will consider a few simpler
cases in $\rr^n$:
\begin{enumerate} 
\item $A$ is a diagonal matrix with entries $\pm 1$ and $S$ is the sphere, 
\item $A$ is diagonal with entries $\pm 1$ and $S$ is any strictly
convex hypersurface.
\end{enumerate} 
We will then consider the general case. The calculations above will be helpful in dealing with this general case.   We will consider
$\pi_L$ only since the result and the proof for $\pi_R$ is the same
as the one in Theorem \ref{thm:k>1 + evals}.

\subsubsection{Special case 1: $A$ diagonal with entries $\pm 1$  and $S$ a  sphere:}

Let us consider the matrix with $k$ entries of $1$ and $n-k$ entries of   $-1$ along the diagonal with $k \in \{2,\dots, n-1\}$.

Let $q$ locally describe the unit sphere. Using coordinates $\vy=(y_1,\cdots, y_{n-1})$ we  have 
\[ q(\vy)=\sqrt{1-y_1^2-\cdots-y_{n-1}^2}.\]
We compute 
\[
\begin{aligned} 
\langle -\n q,1\rangle \delta x&= r_{n}\langle \frac{y_1}{\sqrt{1-|\vy|^2}},\cdots, \frac{y_{n-1}}{\sqrt{1-|\vy|^2}},1\rangle A^{-1}\langle \frac{y_1}{\sqrt{1-|\vy|^2}},\cdots, \frac{y_{n-1}}{\sqrt{1-|\vy|^2}},1\rangle^{T}\\
&=r_n\Big{\{}\frac{y_1^2+\cdots+y_{k}^2-(y_{k+1}^2+\cdots+y_{n-1}^2)}{1-|\vy|^2}-1\Big{\}}.\\
&=r_n\Big{\{}\frac{y_1^2+\cdots+y_{k}^2-(y_{k+1}^2+\cdots+y_{n-1}^2)-1+|\vy|^2}{1-|\vy|^2}\Big{\}}\\
&=r_{n}\Big{\{} \frac{2(y_1^2+\cdots+y_{k}^2)-1}{1-|\vy|^2}\Big{\}}.
\end{aligned} 
\]
In the complement of the set $\{(\vy,\vx,\o) \in \Sigma$ such that $y_1^2+\cdots+y_{k}^2=\frac{1}{2}\}$, where $\Sigma$ is given by \eqref{Eq}, we have fold points as before. 

Let us next consider the set $\Sigma_1=\{(\vy,\vx,\o) \in \Sigma: y_1^2+\cdots+y_k^2=\frac{1}{2}\}$. More explicitly,
\[
\Sigma_1=\Bigg{\{}(\vy, \vx,\o): \langle -\n q,1\rangle\cdot \lb x_1-y_1,\cdots, x_{n-1}-y_{n-1}, x_n-q(\vy)\rb=0; y_1^2+\cdots+y_{k}^2=\frac{1}{2}=y_{k+1}^2+\cdots+y_{n}^2\Bigg{\}}.
\]

As in Theorem \ref{thm:pos def}, we only need to consider a vector field in $\ker(\D
\pi_L)$ of the form $(0, v_1,\cdots, v_{n},0)$. We then have, for each $1\leq i\leq k$, 
\[
v_i=-\frac{y_i}{\sqrt{1-|\vy|^2}} v_{n}
\]
and for $k+1\leq i\leq n-1$, 
\[
v_{i}=\frac{y_i}{\sqrt{1-|\vy|^2}} v_{n}.
\]
Hence $v=(v_1,\cdots, v_n)$ is of the form 
\Beq\label{Eq5.1}
\lb -\frac{y_1}{\sqrt{1-|\vy|^2}}v_n,\cdots, -\frac{y_{k}}{\sqrt{1-|\vy|^2}} v_{n}, \frac{y_{k+1}}{1-|\vy|^2}v_n,\cdots, \frac{y_{n-1}}{1-|\vy|^2}v_n,v_n\rb,
\Eeq 
with $v_n\neq 0$ anywhere.
For $p\in \Sigma_1$, such a $v|_p$ is in $T\Sigma_p$, as can be seen by considering the defining function $K$ from \eqref{Eq} 
and noting that 
\[
\D K (v) = v(K)=\sum\limits_{i=1}^{k} \frac{-y_{i}^2}{1-|\vy|^2} v_{n}+\sum\limits_{i=k+1}^{n-1} \frac{y_{i}^2}{1-|\vy|^2} v_{n}+v_{n}=v_n\lb \frac{1-2|\vy'|^2}{1-|\vy|^2}\rb=0.
\]
Here, $\vy'=(y_1,\cdots,y_k)$.

Next, we check the requirement for a simple zero in the definition of cusp points. Let us take $v$ to be a smooth vector field along $\Sigma$ such that $v\in \ker(\D \pi_L)$ at each point of $\Sigma$. Then $v$ has exactly the form given in \eqref{Eq5.1}. Define 
\[
G(\vy,\vx,\o):=\D K(v)= v(K)=v_n\lb \frac{1-2|\vy'|^2}{1-|\vy|^2}\rb.
\]
Here $v_n$ is a smooth function of $(\vy,\vx,\o)$ non-vanishing. 
Then  for $1\leq i\leq k$, 
\[
\begin{aligned}
\frac{\PD G}{\PD y_i}&=\frac{\PD v_n}{\PD y_i}\frac{1-2|\vy'|^2}{1-|\vy|^2}+v_n \frac{-2y_i +4y_i(y_{k+1}^2+\cdots +y_{n-1}^2)}{(1-|\vy|^2)^2}\\
&=\frac{\PD v_n}{\PD y_i}\frac{1-2|\vy'|^2}{1-|\vy|^2}-2v_n y_i \frac{1-2(y_{k+1}^2+\cdots +y_{n-1}^2)}{(1-|\vy|^2)^2}.
\end{aligned} 
\]
At points $p$ on $\Sigma_1$, the first expression on the right vanishes and  since $y_n>0$ (being described as the graph of the function $f(\vy)=\sqrt{1-|\vy|^2}$), $1-2(y_{k+1}^2+\cdots +y_{n-1}^2)\neq 0$.  Hence at least one of the partial derivatives, $\frac{\PD G}{\PD y_i}$ for $1\leq i\leq k$ cannot vanish at points on $\Sigma_1$, since $y_1^2+\cdots+y_k^2=\frac{1}{2}$. Thus, we get that the set $\Sigma_1$ consists of cusp points for $\pi_L$.

\subsubsection{Special case 2: $A$ diagonal with entries $\pm 1$ and $S$ a smooth strictly convex hypersurface}

Let us  consider the matrix $A$, with $ k$ entries of $1$ and $n-k$ entries of $-1$,  along the diagonal, 
with $k \in \{2,\dots, n-1\}$. 
\[A=\diag(1,\dots,1,-1,\dots,-1).\] Let $S$ be a smooth hypersurface locally given as the graph of the smooth function $q(\vy)$ for $\vy\in \Rb^{n-1}$. 
Then 
\[
\langle -\n q,1\rangle A^{-1}\langle -\n q,1\rangle^{T}=\sum\limits_{i=1}^{k}|\PD_{y_i}q|^2-\sum\limits_{i=k+1}^{n-1}|\PD_{y_i}q|^2-1.
\]
As before, if we exclude the points where $\langle -\n q,1\rangle A^{-1}\langle -\n q,1\rangle^{T}=0$, then, in this set, we have fold points. 

Next, we study the subset $\Sigma_1\subset \Sigma$, where $\Sigma$ is  defined in \eqref{Eq}, by 
\[
\Sigma_1=\{(\vy,\vx,\o) \in \Sigma:  \sum\limits_{i=1}^{k}|\PD_{y_i}q|^2-\sum\limits_{i=k+1}^{n-1}|\PD_{y_i}q|^2=1\}.
\]
We would like to show that $\Sigma_1$ consists of cusp points. We first show that for a point $p\in \Sigma_1$, $\ker(\D \pi_L)_p\subset T\Sigma_p$.  

As in Theorem \ref{thm:pos def}, if  $(\delta y, \delta x, \delta \omega) \in \ker(\D\pi_L)$, then,  $\delta y=0= \delta \omega$,  and we only need to consider a vector field in $\ker(\D \pi_L)$ of the form $(0, v_1,\cdots, v_{n},0)$. We then have, for each $1\leq i\leq k$, 
\[
v_i=v_n\PD_{y_i} q 
\]
and for $k+1\leq i\leq n-1$, 
\[
v_{i}=-v_n \PD_{y_i} q.
\]

Hence $v=(v_1,\cdots, v_n)$ is of the form 
\Beq\label{Eq6.1}
v=(v_n\PD_{y_1} q,\cdots, v_n\PD_{y_k}q,-v_n \PD_{y_{k+1}}q,\cdots,-v_n\PD_{y_{k+1}}q).
\Eeq
For $p\in \Sigma_1$, such a $v\in T\Sigma_p$, since for $K$ the defining function in \eqref{Eq}, we have 
\[
\D K(v)=v(K)=v_n\lb -\sum\limits_{i=1}^{k}|\PD_{y_i} q|^2+\sum\limits_{i=k+1}^{n-1}|\PD_{y_{i}}q|^2+1\rb=0.
\]
Here we note that $v_n$ is a non-vanishing smooth function of $(\vy,\vx,\o)$. 

Next, we check the simple zero condition in the definition of cusp points. Let us take $v$ to be a smooth vector field along $\Sigma$ such that $v\in \ker(\D \pi_L)$ at each point of $\Sigma$. Then $v$ has exactly the form given in \eqref{Eq6.1} with the $\vy$ and $\o$ components being $0$. Define 
\[
G(\vy,\vx,\o):=\D K(v)= v(K)=v_n\lb -\sum\limits_{i=1}^{k}|\PD_{y_i} q|^2+\sum\limits_{i=k+1}^{n-1}|\PD_{y_{i}}q|^2+1\rb.
\]
Here $v_n$ is a smooth non-vanishing function of $(\vy,\vx,\o)$. 
Then  for $1 < k\leq n-1$,

\[
\PD_{y_k}G(\vy,\vx,\o)=\frac{\PD v_n}{\PD y_i}\lb -\sum\limits_{i=1}^{k}|\PD_{y_i} q|^2+\sum\limits_{i=k+1}^{n-1}|\PD_{y_{i}}q|^2+1\rb+2v_{n}\lb -\sum\limits_{i=1}^{k} \frac{\PD^2 q}{\PD y_i \PD y_k}\PD_{y_i}q+\sum\limits_{i=k+1}^{n-1} \frac{\PD^2 q}{\PD y_i \PD y_k}\PD_{y_i}q\rb.
\]
At points on $\Sigma_1$, the first term vanishes, and since $v_n$ is non-vanishing, we have a system of equations of the form 
\[
\begin{pmatrix}
    \frac{\PD^2 q}{\PD y_1^2}&\cdots & \frac{\PD^2 q}{\PD y_1 \PD y_{n-1}}\\
    \vdots &\ddots & \vdots\\
    \vdots &\ddots & \vdots\\
    \vdots &\ddots & \vdots\\
    \frac{\PD^2q}{\PD y_{1} \PD y_{n-1}}&\cdots & \frac{\PD^2 q}{\PD y_{n-1}^2}
\end{pmatrix}\begin{pmatrix}-\PD_{y_1}q\\ \vdots \\-\PD_{y_{k}}q\\\PD_{y_{k+1}}q\\ \vdots\\ \PD_{y_{n-1}}q
\end{pmatrix} = 
\frac{1}{2v_n}\begin{pmatrix} \PD_{y_1} G\\ \vdots \\\vdots \\ \vdots\\ \PD_{y_{n-1}}G
\end{pmatrix}.
\]
If we assume that the function $q$ locally gives the graph of the strictly convex hypersurface $S$,  then the Hessian matrix on the left is positive definite. Also, $\n q$ is non-vanishing, since on the level set we are interested in, that is on $\Sigma_1$, we have 
\[
-\sum\limits_{i=1}^{k}|\PD_{y_i} q|^2+\sum\limits_{i=k+1}^{n-1}|\PD_{y_{i}}q|^2+1=0.
\]
Hence not all components of $\PD_{y_i}q$ can be $0$ on $\Sigma_1$. 
Therefore, $\n G$ is non-vanishing at points on $\Sigma_1$, showing that $\Sigma_1$ consists of cusp points.

\subsubsection{General $A$, $k>1$, $S$ strictly convex
hypersurface}\label{GenA-strictly convex} We now prove the main theorem of this
section, Theorem \ref{thm:k>1 + evals}.

\begin{proof}[Proof of Theorem \ref{thm:k>1 + evals}]
Let $A$ be a general symmetric matrix with $k>1$ positive and $n-k$ negative eigenvalues. 
As before, if we exclude the points where $\langle -\n q,1\rangle A^{-1}\langle -\n q,1\rangle^{T}=0$, then in this set, we have fold points. 

Let us now consider the subset $\Sigma_1\subset \Sigma$, where $\Sigma$ is  defined in \eqref{Eq}, by 
\[
\Sigma_1=\{(\vy,\vx,\o) \in \Sigma:  \langle -\n q,1\rangle A^{-1}\langle -\n q,1\rangle^{T}=0\}.
\]
We would like to show that $\Sigma_1$ consists of cusp points. 
As before, we only need to consider a vector field in $\ker(\D
\pi_L)_{p}$ of the form $(0, v_1,\cdots, v_{n},0)$.
Denoting $v=(v_1,\cdots,v_n)$ and then letting $r=
Av$, we then have, \[ 
\begin{aligned} 
& r_1=- \frac{\partial q}{\partial y_1}  r_n\\
&r_2=- \frac{\partial q}{\partial y_2}  r_n\\
& \vdots \\
&r_{n-1}=- \frac{\partial q}{\partial y_{n-1}}  r_n.
\end{aligned} 
\]
Hence $r$ is of the form 
\Beq\label{Eq6.2}
\vr=(-q_{y_1}r_n,\cdots,-q_{y_{n-1}} r_n, r_n)
\Eeq
and $v=A^{-1}\vr$. 
We note that $v|_p\in T\Sigma_p$, since for $K$ in  \eqref{Eq}, we have 
\[
\D K(v)=v(K)=r_n\lb \langle -\n q,1\rangle A^{-1}\langle -\n q,1\rangle^{T}\rb=0.
\]
We note that $r_n$ is a non-vanishing smooth function of $(\vy,\vx,\o)$.

Next, let us check the simple zero condition in the definition of a cusp. Let us
take $v$ to be a smooth vector field along $\Sigma$ such that $v\in
\ker(\D \pi_L)$ at each point of $\Sigma$. Then $v$ has exactly
the form $A^{-1}r$ with $r$ given in \eqref{Eq6.2} and with the $\vy$ and $\o$ components being $0$. Define 
\[
G(\vy,\vx,\o):=\D K(v)= v(K)=r_n\lb \langle -\n q,1\rangle A^{-1}\langle -\n q,1\rangle^{T}\rb.
\]
Here $r_n$ is a smooth function of $(\vy,\vx,\o)$ non-vanishing. 
Taking derivatives of $G$ with respect to $y_{k}$ for $1\leq k\leq n-1$, we have the following system
\[
\begin{pmatrix}
    \frac{\PD^2 q}{\PD y_1^2}&\cdots & \frac{\PD^2 q}{\PD y_1 \PD y_{n-1}}\\
    \vdots &\ddots & \vdots\\
    \vdots &\ddots & \vdots\\
    \vdots &\ddots & \vdots\\
    \frac{\PD^2q}{\PD y_{1} \PD y_{n-1}}&\cdots & \frac{\PD^2 q}{\PD y_{n-1}^2}
\end{pmatrix}\begin{pmatrix}\lb A^{-1}\rb_{1}\cdot (-\n q,1)\\ \vdots \\\vdots\\ \lb A^{-1}\rb_{n-1}\cdot (-\n q,1)
\end{pmatrix} = 
\frac{1}{2v_n}\begin{pmatrix} \PD_{y_1} G\\ \vdots \\\vdots \\ \vdots\\ \PD_{y_{n-1}}G
\end{pmatrix}.
\]
Here $\lb A^{-1}\rb_i$ refers to the $i^{\mathrm{th}}$ row of $A^{-1}$.
If we assume that the function $q$ locally gives the graph of the strictly convex hypersurface $S$,  then the Hessian matrix on the left is positive definite. Also the column vector 
\Beq\label{Eq6.3} 
\begin{pmatrix}\lb A^{-1}\rb_{1}\cdot (-\n q,1)\\ \vdots \\\vdots\\ \lb A^{-1}\rb_{n-1}\cdot (-\n q,1)
\end{pmatrix}
\Eeq
is non-vanishing since on the level set we are interested in. That is on $\Sigma_1$, we have 
\[
\langle -\n q,1\rangle A^{-1}\langle -\n q,1\rangle^{T}=0,
\]
and if we assume that the above column vector is vanishing, then we would have that $\lb A^{-1}\rb_{n} \cdot (-\n q,1)=0$ as well. Since $A$ is invertible, this would imply that $(-\n q,1)$ is the $0$ vector, which is not possible. 
Hence, not all components of \eqref{Eq6.3} can be $0$ on $\Sigma_1$. 
Therefore, $\n G$ is non-vanishing at points on $\Sigma_1$, showing that $\Sigma_1$ consists of cusp points.

\end{proof}

\subsection{Case $k=1$, $A$ diagonal, entries $\pm
1$}\label{sec:k=1 diag 1}
In this section, we consider $A$ to be diagonal with exactly one entry being $1$ and the rest being $-1$ for the case of $S$ being a sphere. In this case, we can prove that there are no cusp singularities for the left projection.

We consider  the top half of the unit sphere.  Let $\Omega$ be the open unit disk in $\rnm$ and let $q(\vy) = \sqrt{1-y_1^2-\cdots-y_{n-1}^2}$.  Let $\tilde{S}$ be the open top half of the  unit sphere: \[ \tilde{S} = \sparen{\vy, q(\vy)\st \vy\in \Omega}.\]

\begin{theorem}\label{thm:diag strictly cvx}  Let  $\Rc: \Ec'(\rr^n\setminus \tilde{S})
\rightarrow \Dc'(\Omega \times(0,\infty))$ be the Radon transform
defined by \eqref{def:psi} and \eqref{def:R}. Let $A$ be a diagonal
matrix, with one entry ($k=1$) of $1$ and $ n-1$ entries of $ -1$ on
the main diagonal. Then $\Rc$ is an FIO
associated to a two-sided fold canonical relation. 
\end{theorem}

\begin{proof} 
First we assume that the $1$ entry is in the first slot, so $A=\diag(1,-1,\dots,-1)$.  

Repeating the calculations as before, we see from the discussion below  \eqref{DpiL-det-equation} that 
\[
\langle -\n q,1\rangle \delta x=r_n\langle -\n q , 1\rangle A^{-1}\langle -\n q ,1\rangle^{T}.
\]
For the case of $A$ and $q$ at hand, we have 
\[
\begin{aligned} 
\langle -\n q,1\rangle \delta x&=r_n\langle \frac{y_1}{\sqrt{1-|\vy|^2}},\frac{y_2}{\sqrt{1-|\vy|^2}},\cdots,\frac{y_{n-1}}{\sqrt{1-|\vy|^2}},1\rangle\begin{pmatrix}
1 & 0 & \cdots& 0\\
0&-1 & \cdots &0\\
\vdots & \vdots & \ddots & 0\\
0&0&\cdots &-1
\end{pmatrix}\\ &\qquad\qquad\cdot\langle \frac{y_1}{\sqrt{1-|\vy|^2}},\frac{y_2}{\sqrt{1-|\vy|^2}},\cdots,\frac{y_{n-1}}{\sqrt{1-|\vy|^2}},1\rangle^{T}\\
&\qquad\qquad\qquad =r_n\lb \frac{2y_1^2-1}{1-|\vy|^2}\rb.
\end{aligned} 
\]
Consider the intersection of the hyperplanes $y_1=\pm \frac{1}{\sqrt{2}}$ with $\tilde{S}$. In the complement of this intersection, we have $r_n=0$ and following the previous calculation, we have that $ \pi_{L}$ is a fold.

Next, let us consider the points corresponding to $y_1=\pm \frac{1}{\sqrt{2}}$ on $\tilde{S}$. We again recall that $\Sigma$ is characterized by the set of points 
\[\Sigma=\{ (\vy,\vx,\o):  (-(\nabla q)^T(\vy), 1) \cdot (x_1-y_1,\cdots, x_{n-1}-y_{n-1}, x_n-q(y))^T=0    \}.\]
In our setup, we then have
\[
\Sigma=\{((\vy,\vx,\o): (x_1-y_1)\frac{y_1}{\sqrt{1-|\vy|^2}}+\cdots+(x_{n-1}-y_{n-1})\frac{y_{n-1}}{\sqrt{1-|\vy|^2}}+(x_n-\sqrt{1-|\vy|^2})=0\}.
\]
Choosing $y_1=\frac{1}{\sqrt{2}}$, the set $\Sigma$ is characterized by 
\Beq\label{Eq4.3A}
x_1+\sqrt{2}\sum\limits_{i=2}^{n-1}x_i y_i+x_n\sqrt{1-2|\vy'|^2}=\sqrt{2}.
\Eeq
Here $|\vy'|^2=\sum\limits_{i=2}^{n-1} y_i^2$.

We also have that $\vx$ should lie on the surface 
\[
\lb x_1-\frac{1}{\sqrt{2}}\rb^2-(x_2-y_2)^2-\cdots-(x_{n-1}-y_{n-1})^2-\lb x_n-\frac{\sqrt{1-2|\vy'|^2}}{\sqrt{2}}\rb^2=t.
\]
We substitute 
\[
x_1=\sqrt{2}-\sqrt{2}\sum\limits_{i=2}^{n-1}x_i y_i-x_n\sqrt{1-2|\vy'|^2}
\]
into the above equation. We get
\[
\lb \sqrt{2}-\sqrt{2}\sum\limits_{i=2}^{n-1}x_i y_i-x_n\sqrt{1-2|\vy'|^2}-\frac{1}{\sqrt{2}}\rb^2 -(x_2-y_2)^2-\cdots-(x_{n-1}-y_{n-1})^2-\lb x_n-\frac{\sqrt{1-2|\vy'|^2}}{\sqrt{2}}\rb^2=t.
\]
Rearranging this, we have
\[
\lb \frac{1}{\sqrt{2}}-\sqrt{2}\sum\limits_{i=2}^{n-1}x_i y_i-x_n\sqrt{1-2|\vy'|^2}\rb^2 -(x_2-y_2)^2-\cdots-(x_{n-1}-y_{n-1})^2-\lb x_n-\frac{\sqrt{1-2|\vy'|^2}}{\sqrt{2}}\rb^2=t.
\]
Simplifying this, we get,
\[
2\lb \sum\limits_{i=2}^{n-1} x_i y_i\rb^2-2x_n^2|\vy'|^2+2\sqrt{2}x_n \lb \sum\limits_{i=2}^{n-1} x_i y_i\rb\sqrt{1-2|\vy'|^2}-\sum\limits_{i=2}^{n-1} x_i^2=t.
\]
Adding and subtracting $x_n^2$, we then have 
\[
2\lb \sum\limits_{i=2}^{n-1} x_i y_i\rb^2+x_n^2\lb 1-2|\vy'|^2\rb+2\sqrt{2}x_n \lb \sum\limits_{i=2}^{n-1} x_i y_i\rb\sqrt{1-2|\vy'|^2}-\sum\limits_{i=2}^{n} x_i^2=t.
\]
From this, we get,
\Beq\label{Eq4.3AA}
\lb x_n\sqrt{1-2|\vy'|^2}+\sqrt{2}\sum\limits_{i=2}^{n-1} x_i y_i\rb^2-\sum\limits_{i=2}^{n} x_i^2=t.
\Eeq
We can view the first term in \eqref{Eq4.3AA} as a dot product of the two vectors \[A=(x_2,\cdots,x_n)\ \ \text{and}\ \ B=(\sqrt{2}y_2,\cdots,\sqrt{2}y_{n-1},\sqrt{1-2|\vy'|^2}).\] 
We note that 
\[
\lVert B\rVert^2=2y_2^2+ \cdots +2y_{n-1}^2 + 1-2|\vy'|^2=1.
\]

Using Cauchy-Schwartz inequality, we get that $t\leq 0$ in \eqref{Eq4.3AA} and this is a contradiction. Therefore, in this case as well, we don't need to consider the exceptional points on $\tilde{S}$ corresponding to $y_1=\pm \frac{1}{\sqrt{2}}$.

We now prove that the theorem is true when $A$ has the $1$ in the $k\th$
diagonal entry for $k\in\{2,\dots, n-1\}$.  Assume
$A=\diag(-1,\dots,-1,1,-1,\dots,-1)$ where the entry of $1$ is in the
$k^\text{th}$ place. We claim that one can reduce this to the case
$\tilde{A}=\diag(1,-1,\dots,-1)$ and the same manifold $\tilde{S}$,
the top half of the unit sphere. The proof is just a permutation
argument. 

Let $L$ be the $n\times n$ matrix that switches the first row and
$k\th$ row of the identity matrix, and let $L'$ be the $(n-1)\times (n-1)$ matrix obtained
from $L$ by deleting the $n\th$ row and the $n\th$ column of $L$.

We define a change of coordinates on $X$ using $L$ and on $\Omega$ using
$L'$: for $\vx\in X$ and $\vy\in \Omega$, \[\vx\mapsto L\vx=
(x_k,x_2,\dots,x_1,\dots,x_n),\ \ (\vy,t)\mapsto ( L'\vy,t).\]

Therefore, in the new coordinates $A$ becomes $\tilde{A}=LAL =
\diag(1,-1,\dots,-1)$. Because $L$ permutes $x_1$ and $x_k$ and $k<n$,
$q$ does not change, and $L\tilde{S}=\tilde{S}$.

Finally, we need to show that $\pi_L$ and $\pi_R$ have the same
properties in the new coordinates as in the old.  However, this
follows from the fact that the change of coordinates on  $T^*(X)$ is
$(\vx,\xi)\mapsto (L\vx,L\xi)$ and the equivalent change of
coordinates  in $T^*(\Omega)$ is $(\vy,t,\eta',\eta_n)
\mapsto (L'\vy,t,L'\eta',\eta_n)$.

The final case to consider for $\tilde{S}$
is when $A=\diag(-1,\dots,-1,1)$.
For the case of $A$ and $q$ at hand, we have 
\[
\begin{aligned} 
\langle -\n q,1\rangle \delta x&=r_n\langle \frac{y_1}{\sqrt{1-|\vy|^2}},\frac{y_2}{\sqrt{1-|\vy|^2}},\cdots,\frac{y_{n-1}}{\sqrt{1-|\vy|^2}},1\rangle\begin{pmatrix}
-1 & 0 & \cdots& \cdots&0\\
0&-1 & \cdots &\cdots&0\\
0&0&-1 & \cdots &0\\
\vdots & \vdots  &\vdots&\ddots & 0\\
0&0&\cdots &\cdots&1
\end{pmatrix}\\ &\qquad\qquad\cdot\langle \frac{y_1}{\sqrt{1-|\vy|^2}},\frac{y_2}{\sqrt{1-|\vy|^2}},\cdots,\frac{y_{n-1}}{\sqrt{1-|\vy|^2}},1\rangle^{T}\\
&\qquad\qquad\qquad =r_n\lb \frac{1-2|\vy|^2}{1-|\vy|^2}\rb.
\end{aligned} 
\]
Choosing $|\vy|^2=\frac{1}{2}$, the set $\Sigma$ is characterized by 
\[
\sum\limits_{i=1}^{n-1}x_i y_i+\frac{x_n}{\sqrt{2}}=1.
\]
We also have that $\vx$ should lie on the surface 
\[
-\lb x_1-y_1\rb^2-\lb x_2-y_2\rb^2-\cdots-(x_{n-1}-y_{n-1})^2+\lb x_n-\frac{1}{\sqrt{2}}\rb^2=t.
\]
Then we have 
\[
-\lb x_1-y_1\rb^2-\lb x_2-y_2\rb^2-\cdots-(x_{n-1}-y_{n-1})^2+\lb \frac{1}{\sqrt{2}}-\sqrt{2}\sum\limits_{i=1}^{n-1} x_i y_i\rb^2=t.
\]
Expanding, we get,
\[
-\sum\limits_{i=1}^{n-1} x_i^2 -\sum\limits\limits_{i=1}^{n-1} y_i^2+2\sum\limits_{i=1}^{n-1} x_i y_i+\frac{1}{2}+2\lb \sum\limits_{i=1}^{n-1} x_i y_i\rb^2-2\sum\limits_{i=1}^{n-1} x_i y_i =t.
\]
Using the fact that $|\vy|^2=\frac{1}{2}$ and using Cauchy-Schwartz inequality, 
\[
\begin{aligned}
&-\sum\limits_{i=1}^{n-1} x_i^2 -\sum\limits\limits_{i=1}^{n-1} y_i^2+2\sum\limits_{i=1}^{n-1} x_i y_i+\frac{1}{2}+2\lb \sum\limits_{i=1}^{n-1} x_i y_i\rb^2-2\sum\limits_{i=1}^{n-1} x_i y_i\\
&=-\sum\limits_{i=1}^{n-1} x_i^2 +2\lb \sum\limits_{i=1}^{n-1} x_i y_i\rb^2\leq -2\sum\limits_{i=1}^{n-1} x_i^2\leq 0.
\end{aligned} 
\]
This contradicts the fact that $t>0$ and finishes the proof.\end{proof}

\section{Case when $A$ is diagonal, indefinite and $S$ is a cylinder} \label{sec:cylindrical-geom}

\par Without loss of generality, we assume $S$ is the top half of the cylinder defined by \bel{cylinder q}y_n = q(\vy)= \sqrt {1-y_2^2-\cdots-y_{n-1}^2}\ \ \text{where}\ \ \vy = (y_1,\dots,y_{n-1})\in \Omega = \sparen{\vy\in \rnm\st y_2^2+\cdots y_{n-1}^2<1}.\ee As in the previous
section, we distinguish two cases, $k=1$ and $n> k > 1$ and we obtain
similar results.

\subsection{Case $k >1$}\label{sec:k>1 cylinder}

\begin{theorem}\label{thm:k>1 cylinder}  Let  $\Rc: \Ec'(\rr^n\setminus S)
\rightarrow \Dc'(\Omega \times(0,\infty)$ be the Radon transform defined by \eqref{def:psi} and \eqref{def:R}. Let  $A=\diag(1,\dots,1,-1,\dots,-1)$ be the  diagonal matrix with the first $k\in \{2,\dots,n-1\}$  diagonal entries  of $1$ and the rest  $-1$, and $\tilde{S}$ is given by \eqref{cylinder q},  then $\Rc$ is an FIO  associated to a canonical relation with $\pi_L$ a fold singularity and $\pi_R$  a cusp singularity.
\end{theorem}

\begin{proof}
\par
We will let $q$ be an arbitrary smooth function for now and specialize the calculation for the case of a cylinder at the very end. 

Then 
\[
\langle -\n q,1\rangle A^{-1}\langle -\n q,1\rangle^{T}=\sum\limits_{i=1}^{k}|\PD_{y_i}q|^2-\sum\limits_{i=k+1}^{n-1}|\PD_{y_i}q|^2-1.
\]
As before, if we exclude the points where $\langle -\n q,1\rangle A^{-1}\langle -\n q,1\rangle^{T}=0$, then on this set, we have fold points. 

Let us now consider the subset $\Sigma_1\subset \Sigma$, where
$\Sigma$ is defined in \eqref{Eq}, by 
\[
\Sigma_1=\{(\vy,\vx,\o) \in \Sigma:  \sum\limits_{i=1}^{k}|\PD_{y_i}q|^2-\sum\limits_{i=k+1}^{n-1}|\PD_{y_i}q|^2=1\}.
\]
We would like to show that $\Sigma_1$ consists of cusp points for $\pi_L$. We first show that for a point $p\in \Sigma_1$, $\ker(\D \pi_L)_p\subset T\Sigma_p$.  

As before, we only need to consider a vector field in $\ker(\D \pi_L)$ of the form $(0, v_1,\cdots, v_{n},0)$. We then have, for each $1 < i\leq k$, 
\[
v_i=v_n\PD_{y_i} q 
\]
and for $k+1\leq i\leq n-1$, 
\[
v_{i}=-v_n \PD_{y_i} q.
\]
Hence $v=(v_1,\cdots, v_n)$ is of the form 
\Beq\label{Eq6.1A}
v=(v_n\PD_{y_1}q,\cdots, v_n\PD_{y_k}q,-v_n \PD_{y_{k+1}}q,\cdots,-v_n\PD_{y_{k+1}}q).
\Eeq
If such a $v\in T\Sigma_p$, we must have for $K$ the defining
function in \eqref{Eq}, 
\[
\D K(v)=v(K)=v_n\lb -\sum\limits_{i=1}^{k}|\PD_{y_i} q|^2+\sum\limits_{i=k+1}^{n-1}|\PD_{y_{i}}q|^2+1\rb.
\]
But this is $0$. Here we note that $v_n$ is a non-vanishing smooth function of $(\vy,\vx,\o)$. 

Next,  we check the simple zero condition from the definition of a cusp. Let us
take $v$ to be a smooth vector field along $\Sigma$ such that $v\in
\ker(\D \pi_L)$ at each point of $\Sigma$. Then $v$ has exactly the form given in \tred{\eqref{Eq6.2}} with the $\vy$ and $\o$ components being $0$. Define 
\[
G(\vy,\vx,\o):=\D K(v)= v(K)=v_n\lb -\sum\limits_{i=1}^{k}|\PD_{y_i} q|^2+\sum\limits_{i=k+1}^{n-1}|\PD_{y_{i}}q|^2+1\rb.
\]
Here $v_n$ is a smooth function of $(\vy,\vx,\o)$ non-vanishing. 
Then  for $1 < k\leq n-1$,
\[
\PD_{y_k}G(\vy,\vx,\o)=\frac{\PD v_n}{\PD y_i}\lb -\sum\limits_{i=1}^{k}|\PD_{y_i} q|^2+\sum\limits_{i=k+1}^{n-1}|\PD_{y_{i}}q|^2+1\rb+2v_{n}\lb -\sum\limits_{i=1}^{k} \frac{\PD^2 q}{\PD y_i \PD y_k}\PD_{y_i}q+\sum\limits_{i=k+1}^{n-1} \frac{\PD^2 q}{\PD y_i \PD y_k}\PD_{y_i}q\rb.
\]
At points on $\Sigma_1$, the first term vanishes, and since $v_n$ is non-vanishing, we have a system of equations of the form 
\[
\begin{pmatrix}
    \frac{\PD^2 q}{\PD y_1^2}&\cdots & \frac{\PD^2 q}{\PD y_1 \PD y_{n-1}}\\
    \vdots &\ddots & \vdots\\
    \vdots &\ddots & \vdots\\
    \vdots &\ddots & \vdots\\
    \frac{\PD^2q}{\PD y_{1} \PD y_{n-1}}&\cdots & \frac{\PD^2 q}{\PD y_{n-1}^2}
\end{pmatrix}\begin{pmatrix}-\PD_{y_1}q\\ \vdots \\-\PD_{y_{k}}q\\\PD_{y_{k+1}}q\\ \vdots\\ \PD_{y_{n-1}}q
\end{pmatrix} = 
\frac{1}{2v_n}\begin{pmatrix} \PD_{y_1} G\\ \vdots \\\vdots \\ \vdots\\ \PD_{y_{n-1}}G
\end{pmatrix}.
\]
If we assume that the function $q$ locally describes a cylinder,  then the matrix system becomes the following:  
\[
\begin{pmatrix}
   0&\cdots& \cdots & 0\\
    \vdots &\frac{\PD^2 q}{\PD y_2^2} &\ddots & \frac{\PD^2 q}{\PD y_2 y_{n-1}}\\
    \vdots &\vdots &\ddots & \vdots\\
    \vdots &\vdots &\ddots & \vdots\\
    0&\frac{\PD^2 q}{\PD y_2 y_{n-1}} &\cdots& \frac{\PD^2 q}{\PD y_{n-1}^2}
\end{pmatrix}\begin{pmatrix}-\PD_{y_1}q\\ \vdots \\-\PD_{y_{k}}q\\\PD_{y_{k+1}}q\\ \vdots\\ \PD_{y_{n-1}}q
\end{pmatrix} = 
\frac{1}{2v_n}\begin{pmatrix} \PD_{y_1} G\\ \vdots \\\vdots \\ \vdots\\ \PD_{y_{n-1}}G
\end{pmatrix}.
\]

Also, $\n q$ is non-vanishing, since on the level set we are interested in, that is on $\Sigma_1$, we have 
\[
-\sum\limits_{i=2}^{k}|\PD_{y_i} q|^2+\sum\limits_{i=k+1}^{n-1}|\PD_{y_{i}}q|^2+1=0.
\]
Hence not all components of $\PD_{y_i}q$ can be $0$ on $\Sigma_1$. 
Therefore, at least one of $\PD_{y_i} G$ for $2\leq i\leq n$ is non-vanishing at points on $\Sigma_1$, since the reduced matrix on the left obtained by removing the first row and column is positive definite, showing that $\Sigma_1$ consists of cusp points for $\pi_L$. 
\par Next, analyzing the right projection $\pi_R$, we have that this map has a fold singularity. We have as in,  \eqref{Eq 1}  $\frac{1}{\omega} {\vx'}_T^T \hess (q) \vx'_T \mbox{ where } \vx'=(x_1, x_2, \dots, x_{n-1}).$ Since the reduced $\mbox{Hess}(q)$ obtained by removing the first row and column is positive definite, if $x_T' \neq 0$, then we get a fold singularity. Otherwise, if $x_T'=0$, then,  $x_2=y_2,\cdots, x_{n-1}=y_{n-1}$. But this would imply that $x_{n}=q(\vx')$, so that $S$ and $\mbox{supp}(f)$ have non-empty intersection, contradicting our assumption. Hence, we have a fold singularity for the right projection, even in the general case, as in the remark below.

\end{proof}
\begin{remark}
    As it turns out, we did not use the specific structure of a cylinder anywhere. As long as the reduced matrix as specified above is strictly positive definite, the same argument as above goes through. 
\end{remark}

\subsection{Case $k=1$, $S$ a cylinder}\label{sec:k=1 cylinder}
\begin{theorem}\label{thm:k=1 cylinder}  Let  $\Rc: \Ec'(\rr^n\setminus S)
\rightarrow \Dc'(\Omega \times(0,\infty))$ be the Radon transform defined by \eqref{def:psi} and \eqref{def:R}. If $A$ is a diagonal matrix,  with $k=1$ eigenvalue of $1$ and  $ n-1,\ -1$  eigenvalues on the main diagonal, and $S$ is the half-cylinder given by \eqref{cylinder q}, then $\Rc$ is an FIO  associated to a two-sided fold canonical relation. 
\end{theorem}

\begin{proof}

\par Consider a diagonal matrix $A$ with $n-1$ eigenvalues being $-1$ and one eigenvalue being $1$.   Let $S$ be the smooth hypersurface describing a cylinder. We take $S$ locally described by the function $q(\vy)=\sqrt{1-y_2^2-\cdots-y_{n-1}^2}$ in the sense that $S=\{(y_1,y_2,\cdots, y_{n-1},\sqrt{1-y_2^2-\cdots-y_{n-1}^2}\}$ with $y_1$ arbitrary.

Let
\[ A=
\begin{pmatrix}
-1 & 0 & \cdots& 0\\
0&-1 & \cdots &0\\
\vdots & \vdots & \ddots & 0\\
0&0&\cdots &1
\end{pmatrix}.
\]

Then as before 

\[
\begin{aligned} 
\langle -\n q,1\rangle \delta x&=r_n\langle {0,\frac{y_2}{\sqrt{1-|\vy|^2}},\cdots,\frac{y_{n-1}}{\sqrt{1-|\vy|^2}},1}\rangle\begin{pmatrix}
-1 & 0 & \cdots& 0\\
0&-1 & \cdots &0\\
\vdots & \vdots & \ddots & 0\\
0&0&\cdots &1
\end{pmatrix}\\ &\qquad\qquad\cdot\langle 0,\frac{y_2}{\sqrt{1-|\vy|^2}},\cdots,\frac{y_{n-1}}{\sqrt{1-|\vy|^2}},1\rangle^{T} \\
&\qquad\qquad\qquad =r_n  \frac{1-2(\sum\limits_{i=2}^{n-1} y_i^2)}{q^2}  
\end{aligned} 
\]

\par If $ \sum\limits_{i=2}^{n-1} y_i^2\neq 1/2$, then $\pi_L$ is a
fold. If $\sum\limits_{i=2}^{n-1} y_i^2=1/2$, then
$q=\frac{1}{\sqrt{2}}$ and the set $\Sigma$ from \tred{\eqref{Eq}} becomes

\Beq\label{Eqs5.1}
\begin{aligned} 
\Sigma&=
\{ (\vy,\vx,\o):  (-(\nabla q)^T(y), 1) \cdot (x_1-y_1,\cdots, x_{n-1}-y_{n-1}, x_n-q(y))^T=0   \}\\
&=(x_2-y_2)\frac{y_2}{q} + \dots + (x_{n-1}-y_{n-1})\frac{y_{n-1}}{q} + (x_n-q) =0\\
&=\Bigg{\{} \sum\limits_{i=2}^{n-1} x_i y_i+ x_n q=1\Bigg{\}}.
\end{aligned} 
\Eeq
We consider the points on the surface: $\{ (x_n-q)^2-(x_{n-1}-y_{n-1})^2- \dots (x_1-y_1)^2=t \}$ for some $t>0$. The surface equation is then  
 \Beq\label{Eqs5.2} 
 \begin{aligned} 
 &\Bigg{\{} x_n^2-2x_nq+q^2-\sum\limits_{i=2}^{n-1} x_i^2 +2 \sum\limits_{i=2}^{n-1} x_iy_i -\sum\limits_{i=2}^{n-1}y_i^2-(x_1-y_1)^2=t  \Bigg{\}}\\
 &= \Bigg{\{}  x_n^2-4x_nq-\sum\limits_{i=2}^{n-1} x_i^2 +2-(x_1-y_1)^2=t\Bigg{\}}\\
 &=\Bigg{\{} (x_n-\sqrt{2})^2-\sum\limits_{i=2}^{n-1} x_i^2 -(x_1-y_1)^2=t\Bigg{\}}.
 \end{aligned} 
 \Eeq 
 From  \eqref{Eqs5.1}, 
 \Beq\label{Eq5.3}
 \frac{x_n}{\sqrt{2}}=1-\sum\limits_{i=2}^{n-1} x_i y_i.
 \Eeq
Slightly rewriting \eqref{Eqs5.2}, we get, 
 \[
 2\lb \frac{x_n}{\sqrt{2}}-1\rb^2-\sum\limits_{i=2}^{n-1} (x_1-y_1)^2=t. 
 \]
 Substituting \eqref{Eq5.3} above, we get, 
 \[
 2\lb \sum\limits_{i=2}^{n-1} x_i y_i\rb^2-\sum\limits_{i=2}^{n-1} x_i^2 -(x_1-y_1)^2=t. 
 \]
 Using the Cauchy-Schwarz inequality, we have 
 \[
 2\lb \sum\limits_{i=2}^{n-1} x_i y_i\rb^2\leq 2 \lb \sum\limits_{i=2}^{n-1} x_i^2\rb \lb \sum\limits_{i=2}^{n-1} y_i^2\rb =\sum\limits_{i=2}^{n-1} x_i^2,
 \]
 since $\sum\limits_{i=2}^{n-1} y_i^2=\frac{1}{2}$. 

Therefore,
\[
2\lb \sum\limits_{i=2}^{n-1} x_i y_i\rb^2-\sum\limits_{i=2}^{n-1} x_i^2 \leq 0. 
\]
This implies that $t\leq 0$, which is a contradiction. 

Locally, we have expressed the cylinder such that the $n^{\mathrm{th}}$ coordinate is written as a function of the $n-2$ variables $y_2,\cdots, y_{n-1}$. But one should consider other coordinates as well. Equivalently, we can interchange the coordinates if necessary so that $y_n$ is always written as the function of the variables $y_2,\cdots,y_{n-1}$. But then, we would have to consider $A$ with $1$ at the other diagonal entries. 

For instance, let us consider $A$ to be 
\[ A=
\begin{pmatrix}
1 & 0 & \cdots& 0\\
0&-1 & \cdots &0\\
\vdots & \vdots & \ddots & 0\\
0&0&\cdots &-1
\end{pmatrix}.
\]

In this case 
\[
\begin{aligned} 
\langle -\n q,1\rangle \delta x&=r_n\langle {0,\frac{y_2}{\sqrt{1-|\vy|^2}},\cdots,\frac{y_{n-1}}{\sqrt{1-|\vy|^2}},1}\rangle\begin{pmatrix}
1 & 0 & \cdots& 0\\
0&-1 & \cdots &0\\
\vdots & \vdots & \ddots & 0\\
0&0&\cdots &-1
\end{pmatrix}\\ &\qquad\qquad\cdot\langle 0,\frac{y_2}{\sqrt{1-|\vy|^2}},\cdots,\frac{y_{n-1}}{\sqrt{1-|\vy|^2}},1\rangle^{T} \\
&\qquad\qquad\qquad =-\frac{r_n}{q^2}. 
\end{aligned}
\]
Since this is $0$ if and only if $r_n=0$, we get a fold singularity. 
\par Using the same permutation argument as above in section 5.2,  it is enough to consider the following matrix: 

\[ A=
\begin{pmatrix}
-1 & 0 & \cdots& \cdots & 0\\
0&1 & \cdots &\cdots &0\\
0&0&-1&\cdots &0\\
\vdots & \vdots & \ddots &\vdots& 0\\
0&0&\cdots &\cdots &-1
\end{pmatrix}.
\]
In this case 
\[
\begin{aligned} 
\langle -\n q,1\rangle \delta x&=r_n\langle {0,\frac{y_2}{\sqrt{1-|\vy|^2}},\cdots,\frac{y_{n-1}}{\sqrt{1-|\vy|^2}},1}\rangle\begin{pmatrix}
-1 & 0 & \cdots& \cdots & 0\\
0&1 & \cdots &\cdots &0\\
0&0&-1&\cdots &0\\
\vdots & \vdots & \ddots &\vdots& 0\\
0&0&\cdots &\cdots &-1
\end{pmatrix}\\ &\qquad\qquad\cdot\langle 0,\frac{y_2}{\sqrt{1-|\vy|^2}},\cdots,\frac{y_{n-1}}{\sqrt{1-|\vy|^2}},1\rangle^{T} \\
&\qquad\qquad\qquad =r_n\frac{2y_2^2-1}{q^2}. 
\end{aligned}
\]
We let $y_2=\pm \frac{1}{\sqrt{2}}$. 
\[
\begin{aligned} 
\Sigma&=
\{ (\vy,\vx,\o):  (-(\nabla q)^T(y), 1) \cdot (x_1-y_1,\cdots, x_{n-1}-y_{n-1}, x_n-q(y))^T=0   \}\\
&=(x_2-y_2)\frac{y_2}{q} + \dots + (x_{n-1}-y_{n-1})\frac{y_{n-1}}{q} + (x_n-q) =0\\
&=\Bigg{\{} \sum\limits_{i=2}^{n-1} x_i y_i+ x_n q=1\Bigg{\}}.
\end{aligned} 
\]
We consider the points on the surface: $\{ (x_2-y_2)^2-(x_1-y_1)^2-(x_3-y_3)^2-\cdots-(x_{n-1}-y_{n-1})^2- (x_n-q)^2=t, t>0 \}$. 
Letting $y_2=\pm \frac{1}{\sqrt{2}}$, we have 
\[
\lb x_2\pm \frac{1}{\sqrt{2}}\rb^2-\sum\limits_{i=3}^{n-1} x_i^2-\sum\limits_{i=3}^{n-1} y_i^2+2\sum\limits_{i=3}^{n-1} x_i y_i -x_n^2-q^2+2x_n q-(x_1-y_1)^2=t.
\]
For simplicity, we consider $y_2=\frac{1}{\sqrt{2}}$. The same argument below works for the other case as well. Using the fact that $\sum\limits_{i=3}^{n-1} x_i y_i +x_n q=1-x_2 y_2=1-\frac{x_2}{\sqrt{2}}$, the above equality simplifies to 
\[
\lb x_2- \frac{1}{\sqrt{2}}\rb^2-\sum\limits_{i=3}^{n-1} x_i^2-\sum\limits_{i=3}^{n-1} y_i^2 +2-\sqrt{2}x_2 -x_n^2-q^2-(x_1-y_1)^2=t.
\]
Using the expression for $q^2$ combined with the fact that $y_2^2=\frac{1}{2}$, we get,
\Beq\label{Eq5.5}
(x_2-\sqrt{2})^2-\sum\limits_{i=3}^{n-1} x_i^2-x_n^2-(x_1-y_1)^2=t.
\Eeq
Using 
\[
\sum\limits_{i=2}^{n-1} x_i y_i + x_n q=1,
\]
we get
\[
\lb x_2-\sqrt{2}\rb^2 =2\lb x_n q +\sum\limits_{3}^{n-1} x_i y_i\rb^2.
\]
By the Cauchy-Schwarz inequality, we get
\[
\lb x_n q +\sum\limits_{3}^{n-1} x_i y_i\rb^2\leq \lb \sum\limits_{i=3}^{n} x_i^2 \rb \lb \sum\limits_{i=3}^{n-1} y_i^2 + q^2\rb=\lb \sum\limits_{i=3}^{n} x_i^2 \rb y_2^2= \frac{1}{2} \lb \sum\limits_{i=3}^{n} x_i^2 \rb.
\]
Therefore, we get, 
\[
(x_2-\sqrt{2})^2 \leq \sum\limits_{i=3}^{n} x_i^2.
\]
This implies $t\leq 0$ in \eqref{Eq5.5}, which is a contradiction. 
\end{proof}

\section{Paraboloids and more general  surfaces of
integration}\label{sec:paraboloids}

So far, we have considered surfaces defined by homogeneous second-order polynomials $\vxtt A \vxt=t$ where $A$ is symmetric and
invertible. In $\rthree$, this includes all non-degenerate real
quadrics except paraboloids. In this section, we study transforms
defined by more general second-order polynomials that are not
homogeneous. Let $k\in \{1,\dots,n-1\}$ and for $\vz\in \rn$, let
$\vz'=(z_1,\dots,z_k)$ and $\vz'' = (z_{k+1},\dots,z_n)$. We let $A'$
be a symmetric invertible $k\times k$ matrix and
$\vb=(b_{k+1},\dots,b_n) \in \rr^{n-k}\setminus \{\zero\}$ and
consider surface defined by \bel{more general} t-(\vxtt)' A'(\vxt)' +
\vb^T(\vxt)'' =0.\ee

We consider an example after stating and proving the main theorem of this
section.

\begin{theorem}[The case $k=n-1$]\label{thm:k=n-1} Let $A'$ be an $(n-1)\times (n-1)$ symmetric
invertible matrix and   let $b\in \rr\setminus \{0\}$.   Let $S$ be a
smooth surface in $\rn$ parametrized by $y_n = q(\vy)$ where
$q:\Omega\to\rr$, and $\Omega$
is an open subset of $\rr^{n-1}$.  Then, the Radon transform $\Rc:\Ec'(\rn\setminus S)\to \Dc'(\Omega\times (0,\infty))$ on surfaces
\bel{k=n-1} t-(\vxtt)'A'(\vxt)' -b(x_n-q) =0\ee with centers on
$S$ satisfies the Bolker condition.\end{theorem}

Note that the center set $S$ is an arbitrary smooth manifold
parameterized as a graph. The condition $b\neq 0$ is important; if
$b=0$, then the manifolds of integration are cylindrical surfaces, and
therefore, they only detect wavefront conormal to the axis of the
cylinder, the $x_n-$axis. The Bolker condition does not hold since
$x_n$ cannot be determined from $\pi_L$, as one can see from equation
\eqref{PiL general 1}.

  \begin{proof} Using the notation in the theorem, we see that the
phase function  of our transform is  \[\psi(\vy,t,\vx, \omega) = 
(t-(\vxtt)'A'(\vxt)' -b(x_n-q))\omega,\] and our surface is given by
$\psi(\vy,t, \vx, \omega) = 0$. Using this phase function, we see that the canonical relation is
\[C = \sparen{\paren{\vy,t, \omega(2 A'(\vxt)' +b\nabla q),\omega; \vx,
\omega\begin{pmatrix}2A'(\vxt)'\\ b\end{pmatrix}}\st \vy\in S,\
t=(\vxtt)'A'(\vxt)'  +b(x_n-q),
\ \omega \neq 0 }.\] In coordinates $(\vy,\vx,\omega)$, we see
\bel{PiL general 1}
\pi_{L}(\vy,\vx,\omega) = (\vy,(\vxtt)'A'(\vxt)' +b(x_n-q),
\omega(2A'(\vxt)' + b\nabla q), \omega)=:(\vy,t,\vz,\tau)\in
T^*(\Omega \times \rr).\ee

We now explain why $\pi_L$ is injective. The first coordinate of this
projection gives $\vy$ and therefore $q(\vy)$ and $\nabla q(\vy)$. The
last coordinate, $\tau$, specifies $\omega$, and the second to last
($\vz$ coordinate) gives $(\vxt)'$ since $A'$ is invertible. Finally,
the value of $x_n$ is given by the $t$ coordinate since $b\neq 0$.
This shows $\pi_L$ is injective.

The formulas that determine $\vx',\ x_n$, and $\omega$ from
coordinates $(\vy,t,\vz,\tau)$ are all smooth functions, so
$\pi_L^{-1}$ is also smooth. Then, $\pi_L$ is an immersion since
$\D\pi_L^{-1}\D \pi_L = I_n$ by the chain rule. Therefore, this transform
satisfies the Bolker condition.

\end{proof}

\begin{remark}[The case $k<n-1$]\label{rem:k<n-1} Now let $k\in \{1,\dots,(n-2)\}$. 
We get a completely different result. Following the notation at the
start of this section, including \eqref{more general},
and using arguments  similar to those in the proof of Theorem
\ref{thm:k=n-1}, we get
\[\pi_L(\vy,\vx,\omega) = (\vy,(\vxtt)'A'(\vxt)' +\vb^T[\vx'' -
(y_{k+1},\dots,y_{n-1},q)], \omega(\vw+b_n\nabla q) , \omega)
\]
where $\vw=\paren{2\paren{A'(\vxt)'}^T,b_{k+1}, \dots, b_{n-1}}^T$. As in
proof of the theorem, $\vy, \omega$ and $\vx'$ can be determined from
this projection.

Next, $\vb^T(\vxt)''=t-(\vxtt)'A'(\vxt)'$ is determined from the $t$
coordinate of the projection. Since $k<n-1$, this does not determine
$\vx''$ and the Bolker condition fails. For this same reason, the
projection $\pi_L$ from the canonical relation drops rank by $n-k-1$
since $\vb\neq \zero$. 
Therefore, if $k<n-1$, then the transform can be quite degenerate. 
\end{remark}

\begin{example}
 To show what happens in a specific case, we go through the
 calculations for a paraboloid in $\rn: (x_1-y_1)^2+(x_2-y_2)^2+\dots +
 (x_{n-1}-y_{n-1})^2-(x_n-q)=t$ (in the notation of  \eqref{k=n-1}, $A'=I_{n-1}$ and
 $b=-1$). In this case, the phase function $\psi$ becomes
 \[
 \begin{aligned} 
 \psi(\vy, t,\vx, \omega)=&(t-(\vxtt)'(\vxt)'+ (x_n-q(\vy)))\omega\\ 
 &\ \ \ =(t-(x_1-y_1)^2-(x_2-y_2)^2-\dots -( x_{n-1}-y_{n-1})^2+(x_n-q(\vy)))\omega.
 \end{aligned} 
 \]
 The canonical relation in this case is 
 \[
 \begin{aligned} 
 C=&\Bigg{\{}\Big{(}\vy,
 \sum\limits_{i=1}^{n-1}(x_i-y_i)^2-(x_n-q(\vy)), \omega(2
 (x_1-y_1)-q_1', 2 (x_2-y_2)-q_2', \cdots, 2(x_{n-1}-y_{n-1})-q_{n-1}'),\omega; \\
 &\qquad\vx,  \o(2(x_1-y_1),\cdots,2(x_{n-1}-y_{n-1}, -1)\Big{)}\Bigg{\}}
 \end{aligned}
 \] 
 and $\pi_L$ becomes:
 \[
 \begin{aligned} 
 \pi_L(\vy, \vx, \omega)=(&\vy, \sum_{i=1}^{n-1}
 (x_i-y_i)^2-(x_n-q(\vy)),\\ & \omega(2(x_1-y_1)-q_1', 2 (x_2-y_2)-q_2',
 \cdots, 2 (x_{n-1}-y_{n-1})-q_{n-1}'),\omega).
 \end{aligned}
 \] Then, it is easy to see that $\vy$ is given by the first
 coordinate of $\pi_L$, $\omega$ is given by the last
 coordinate, $\vx'$ is given by the second-to-last coordinate, and
 $x_n$ is given by the $t-$coordinate.  One finishes the proof as in
the proof of  Theorem 
 \ref{thm:k=n-1}.

 The same analysis holds for any surface of the form $x_1^2 \pm x_2^2
\pm \dots \pm x_{n-1}^2-x_n=t$. \end{example}

\section{The normal operator $\Rc^*\Rc$}\label{sec:normal-operator}

\par Now, we consider the normal operator $\Rc^*\Rc$ from the previous
cases. In some cases that $\Rc^*$ cannot be composed with $\Rc$
without a cutoff $\varphi$ in the middle: $\Rc^*\varphi R$, for example, in some cases when $S$  is unbounded. In such cases, we will let $\Rc^*$ be the
formal dual multiplied on the right by $\varphi$.

We give a brief summary about the $I^{p,l}$ classes. They were
first introduced by Melrose and Uhlmann \cite{MU}, Guillemin and Uhlmann \cite{Guh}
and Greenleaf and Uhlmann \cite{GU1990a}.

 \begin{definition} [] Two submanifolds $M$ and $N$
intersect {\it cleanly} if $M \cap N$ is a smooth submanifold and $T(M
\cap N)=TM \cap TN$.
\end{definition}

\begin{definition} \cite{Guh}
$S^{p,l}(m,n,k)$ is the set of all functions $a(z;\xi,\sigma) \in
C^{\infty} (\rr^m \times \rr^n \times \rr^k )$ such that for every $K
\subset \rr^m$ and every $\alpha \in \zz^m_+, \beta \in \zz^n_+, \delta \in
\zz^k_+$ there is $c_{K, \alpha, \beta}$ such that
\[|\partial_z^{\alpha}\partial_{\xi}^{\beta}\partial_{\sigma}^{\delta}
a(z,\xi,\sigma)| \le c_{K,\alpha,\beta}(1+ |\xi|)^{p- |\beta|} (1+
|\sigma|)^{l-| \delta|}\] for all $(z,\xi,\tau) \in K \times
\rr^n \times \rr^k$.
\end{definition}

\begin{example} Let
$\tilde{\Lambda}_0=\Delta_{T^*\rr^n}=\{ (x, \xi; x, \xi)\st x \in \rr^n, \
\xi \in \rr^n \setminus 0 \}$ be the diagonal in $T^*\rr^n \times T^*\rr^n$
and let $\tilde{\Lambda}_1= \{ (x', x_n, \xi', 0; x', y_n, \xi', 0)\st
x' \in \rr^{n-1}, \ \xi' \in \rr^{n-1} \setminus 0 \}$.  Then,
$\tilde{\Lambda}_0$ intersects $ \tilde{\Lambda}_1$ cleanly in
codimension $1$.
\end{example}

\begin{definition}\cite{Guh}   Let $I^{p,l}(\tilde{\Lambda}_0,
\tilde{\Lambda}_1)$ be the set of all distributions $u$ such that
$u=u_1 + u_2$ with $u_1 \in C^{\infty}_0$ and $$u_2(x,y)=\int
e^{i((x'-y')\cdot \xi'+(x_n-y_n-s) \cdot \xi_n+ s \cdot \sigma)}
a(x,y,s; \xi,\sigma)d\xi d\sigma ds$$ with $a \in S^{p',l'}$ where
$p'=p-\frac{n}{2}+\frac{1}{2}$ and $l'=l-\frac{1}{2}$.
\end{definition}

Let $\Lambda_0$ and $\Lambda_1$ be Lagrangian
submanifolds of the product space $T^*X \times T^*Y$ which intersect cleanly. Then, there is microlocally, a canonical transformation $\chi$ which maps
$(\Lambda_0, \Lambda_1)$ into $(\tilde{\Lambda}_0,
\tilde{\Lambda}_1)$. Thus, we define the $I^{p,l}(\Lambda_0, \Lambda_1)$ class for
any two cleanly intersecting Lagrangians in codimension $1$ as:

\begin{definition} \cite{Guh} Let  $I^{p,l}( \Lambda_0, \Lambda_1)$
be the set of all distributions $u$ such that $ u=u_1 + u_2 + \sum
v_i$ where $u_1 \in I^{p+l}(\Lambda_0 \setminus \Lambda_1)$, $u_2 \in
I^{p}(\Lambda_1 \setminus \Lambda_0)$, the sum $\sum v_i$ is locally
finite and $v_i=Aw_i$ where $A$ is a zero-order FIO associated to
$\chi ^{-1}$, the canonical transformation from above, and $w_i \in
I^{p,l}(\tilde {\Lambda}_0, \tilde{\Lambda}_1)$.
\end{definition}
\begin{theorem}
\cite{GU1990a} If $u \in D'(X \times Y )$ then $u \in I^{p,l}(\Lambda_0,\Lambda_1)$ if there is an $s_0 \in R$ such that for all first order pseudodifferential operators $ P_i$ with principal symbols vanishing on $\Lambda_0 \cup \Lambda_1$, we have $P_1P_2 . . . P_ru \in H_{loc}^{s_0}$.
\end{theorem}

\par The composition $\Rc^*\Rc$ depends on the singularities of $\pi_L$
and $\pi_R$. In Sections \ref{sec:definite}, \ref{sec:k=1 diag 1}, and \ref{sec:k=1 cylinder}, we showed that both projections
have fold singularities. Then, by \cite{RF1}, \cite{Nolan-Fold2000}, $\Rc^*\Rc \in I^{p,l}(\Delta, C)$
where $C$ is another two-sided fold. In sections \ref{sec:k>1 indefinite} and \ref{sec:k>1 cylinder},  we have that $\pi_L$ has a fold singularity and $\pi_R$ has a cusp singularity. In this case, no result for $\Rc^*\Rc$ is known, and this is part of future work. But we would like to note that, in this case, the distribution class associated to $\Rc\Rc^*$ is known from \cite{FN:fold/cusp}. It is an operator with the kernel in the class $
I^m(\Delta, \tilde{C})$ where $\tilde{C}$ is an open umbrella. 
Finally, in section \ref{sec:paraboloids}, we showed that the projections do not drop rank, and satisfy the Bolker condition, thus $\Rc^*\Rc$ is a pseudodifferential operator. 

\section{Acknowledgements}\label{sec:acknowledgements}
The authors began this research as a part of a collaborate@ICERM program at Brown University in 2024.  We thank ICERM for its hospitality and for providing pleasant, productive working conditions. ETQ acknowledges partial support of the Simons Foundation under grant  [MPS-TSM-00708556, ETQ]. VPK acknowledges the support of the Department of Atomic Energy,  Government of India, under
Project No.\ 12-R\&D-TFR-5.01-0520.

\bibliographystyle{plain}
%\bibliography{references}

\end{document}